\documentclass[a4paper]{amsart}
\usepackage{amsthm,amsfonts,amsmath,amssymb}
\usepackage[abs]{overpic}

\newtheorem{theorem}{Theorem}[section]
\newtheorem{lemma}[theorem]{Lemma}

\newtheorem{proposition}[theorem]{Proposition}
\newtheorem{corollary}[theorem]{Corollary}

\theoremstyle{definition}
\newtheorem{definition}[theorem]{Definition}

\newtheorem{remark}[theorem]{Remark}

\theoremstyle{remark}

\newcommand{\Z}{\mathbb{Z}}

\makeatletter

\@addtoreset{figure}{section}
\makeatother

\makeatletter
  
  \@addtoreset{equation}{section}
\makeatother

\begin{document}
\title{Virtualized Delta moves for virtual knots and links}

\author{Takuji NAKAMURA}
\address{Faculty of Education, 
University of Yamanashi,
Takeda 4-4-37, Kofu, Yamanashi, 400-8510, Japan}
\email{takunakamura@yamanashi.ac.jp}

\author{Yasutaka NAKANISHI}
\address{Department of Mathematics, Kobe University, 
Rokkodai-cho 1-1, Nada-ku, Kobe 657-8501, Japan}
\email{nakanisi@math.kobe-u.ac.jp}

\author{Shin SATOH}
\address{Department of Mathematics, Kobe University, 
Rokkodai-cho 1-1, Nada-ku, Kobe 657-8501, Japan}
\email{shin@math.kobe-u.ac.jp}

\author[Kodai Wada]{Kodai Wada}
\address{Department of Mathematics, Kobe University, Rokkodai-cho 1-1, Nada-ku, Kobe 657-8501, Japan}
\email{wada@math.kobe-u.ac.jp}

\makeatletter
\@namedef{subjclassname@2020}{%
  \textup{2020} Mathematics Subject Classification}
\makeatother


\thanks{This work was supported by JSPS KAKENHI Grant Numbers 
JP20K03621, JP19K03492, JP22K03287, and JP23K12973.}



\begin{abstract}
We introduce a local deformation called the virtualized $\Delta$-move for virtual knots and links. 
We prove that the virtualized $\Delta$-move is an unknotting operation for virtual knots. 
Furthermore we give a necessary and sufficient condition 
for two virtual links to be related by a finite sequence of virtualized $\Delta$-moves. 
\end{abstract}

\maketitle

\section{Introduction} 

Although the crossing change is elemental among local deformations 
in classical knot theory, 
the virtualization replacing a real crossing with a virtual crossing 
is considered as a more elemental deformation 
in virtual knot theory; 
in fact, a crossing change is realized by two virtualizations. 

In this paper, we will introduce an elemental version of a $\Delta$-move. 
Here, the $\Delta$-move is one of the important local deformations 
in classical knot theory. 
In fact, 
it is an unknotting operation for classical knots, 
and characterizes classical links with the same pairwise linking numbers.

\begin{definition}\label{def-vd}
A {\it virtualized $\Delta$-move} is a local deformation 
on a virtual link diagram 
as shown in {\rm Figure~\ref{virtualized-delta}}. 
We denote it by $v\Delta$ in figures. 
\end{definition}

\begin{figure}[htbp]
  \centering
    \begin{overpic}[width=8cm]{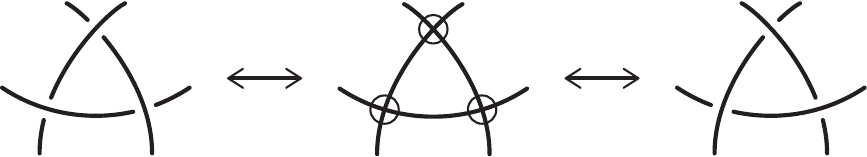}
      \put(63,25){$v\Delta$}
      \put(152,25){$v\Delta$}
    \end{overpic}
  \caption{A virtualized $\Delta$-move}
  \label{virtualized-delta}
\end{figure}

By definition, 
we see that a $\Delta$-move is realized by a combination of two 
virtualized $\Delta$-moves and 
a generalized Reidemeister move VI. 
In this sense, 
we can say that the virtualized $\Delta$-move 
is more elemental than the $\Delta$-move.

\begin{definition}
Two virtual links $L$ and $L'$ are 
{\it $v\Delta$-equivalent} to each other if their diagrams 
are related by a finite sequence of 
virtualized $\Delta$-moves and 
generalized Reidemeister moves. 
\end{definition}

For virtual knots, 
we have the following.

\begin{theorem}\label{thm13}
Any virtual knot is $v\Delta$-equivalent 
to the trivial knot; 
that is, the virtualized $\Delta$-move 
is an unknotting operation 
for virtual knots. 
\end{theorem}

Let ${\rm u}_{v\Delta}(K)$ 
be the minimal number of virtualized $\Delta$-moves 
which needs to deform $K$ into the trivial knot.

\begin{theorem}\label{thm14}
For any integer $m\geq 1$, 
there are infinitely many virtual knots $K$ 
with ${\rm u}_{v\Delta}(K)=m$. 
\end{theorem}

For an $n$-component virtual link $L$ with $n\geq2$, 
we will define invariants $p_i(L)\in\{0,1\}$ 
($i=1,\dots,n$). 
Then we have the following.

\begin{theorem}\label{thm15}
For $n\geq 2$, 
two $n$-component virtual links $L$ 
and $L'$ are $v\Delta$-equivalent to each other 
if and only if $p_i(L)=p_i(L')$ holds for any $i=1,\dots,n$. 
\end{theorem}

This paper is organized as follows. 
In Section~\ref{sec2}, 
we study several properties of 
virtualized $\Delta$-moves 
for virtual knots, 
and prove Theorems~\ref{thm13} 
and \ref{thm14}. 
In Section~\ref{sec3}, 
we study the behavior of virtualized $\Delta$-moves 
for virtual links, 
and prove Theorem~\ref{thm15}. 
In Section~\ref{sec4}, 
we divide virtualized $\Delta$-moves 
into four types, 
and prove that any one of them generates 
the other three.

\section{Virtualized $\Delta$-moves for virtual knots}\label{sec2}

\begin{lemma}\label{lem-cc}
A crossing change at a real crossing 
is realized by a combination of 
a virtualized $\Delta$-move and 
generalized Reidemeister moves. 
\end{lemma}

\begin{proof}
This follows from Figure~\ref{pf-lem-cc}, 
where the symbol $\overset{\textrm{R}}{\longleftrightarrow}$ means  
a combination of generalized Reidemeister moves. 
\end{proof}

\begin{figure}[htbp]
  \centering
    \begin{overpic}[width=12cm]{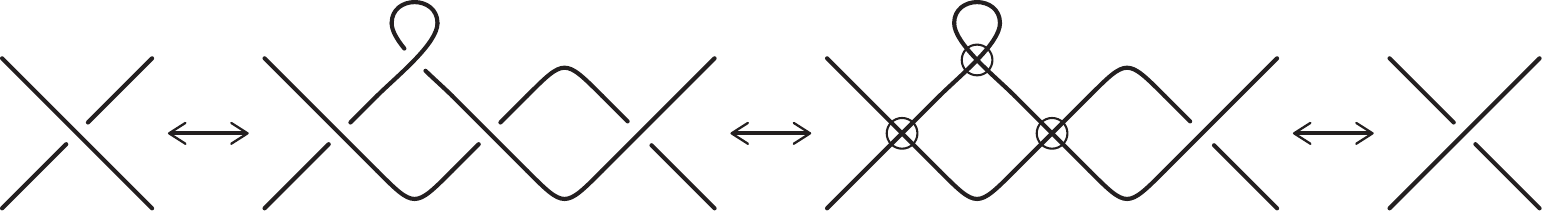}
      \put(42.5,21){R}
      \put(164,21){$v\Delta$}
       \put(291.5,21){R}
    \end{overpic}
  \caption{Proof of Lemma~\ref{lem-cc}}
  \label{pf-lem-cc}
\end{figure}

\begin{lemma}\label{lem-fd}
A local deformation FD as shown in 
{\rm Figure~\ref{forbidden-detour}} is realized by 
a combination of 
two virtualized $\Delta$-moves and 
generalized Reidemeister moves. 
\end{lemma}

\begin{figure}[htbp]
  \centering
    \begin{overpic}[width=4.5cm]{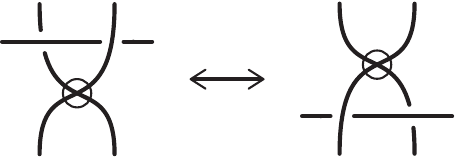}
      \put(57,26.5){FD}
    \end{overpic}
  \caption{A local deformation FD}
  \label{forbidden-detour}
\end{figure}

\begin{proof}
This follows from Figure~\ref{pf-lem-fd}.
\end{proof}

\begin{figure}[htbp]
  \centering
    \begin{overpic}[width=12cm]{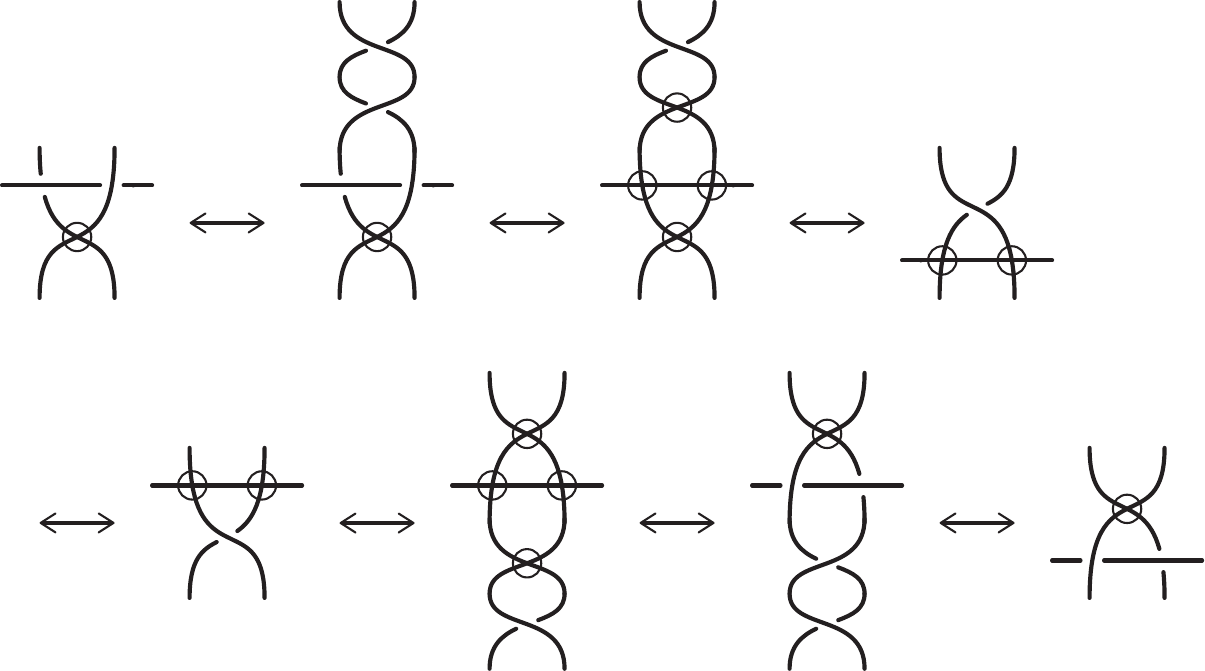}
      \put(60.7,132){R}
      \put(142.7,132){$v\Delta$}
      \put(230.7,132){R}
      \put(18,47){R}
      \put(103,47){R}
      \put(185,47){$v\Delta$}
      \put(273,47){R}
    \end{overpic}
  \caption{Proof of Lemma~\ref{lem-fd}}
  \label{pf-lem-fd}
\end{figure}

\begin{remark} 
The local deformation FD in Figure~\ref{pf-lem-fd} 
is called a {\it forbidden detour move}~\cite{CMG,YI} 
or a {\it fused move}~\cite{ABMW}. 
\end{remark}

\begin{lemma}\label{lem-forbidden}
A forbidden move 
is realized by a combination of 
virtualized $\Delta$-moves and 
generalized Reidemeister moves. 
\end{lemma}

\begin{proof}
The sequence of Figure~\ref{pf-lem-forbidden} shows that 
an upper forbidden move is realized by a combination of 
two crossing changes and a forbidden detour move, 
where the symbol $\overset{\textrm{cc}}{\longleftrightarrow}$ means 
a combination of crossing changes. 
By Lemmas~\ref{lem-cc} and \ref{lem-fd}, 
we have the result for the case of an upper forbidden move. 
The case of a lower forbidden move is proved similarly. 
\end{proof}

\begin{figure}[htbp]
  \centering
    \begin{overpic}[width=10cm]{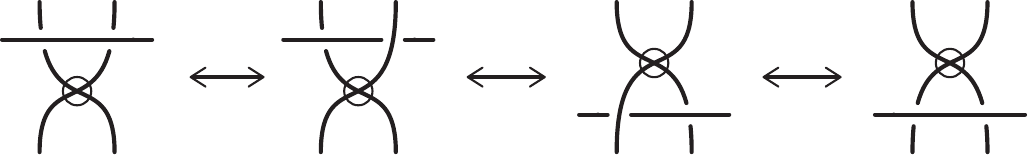}
      \put(58.5,26){cc}
      \put(133.3,26){FD}
      \put(217,26){cc}
    \end{overpic}
  \caption{An upper forbidden move}
  \label{pf-lem-forbidden}
\end{figure}

\begin{proof}[Proof of {\rm Theorem~\ref{thm13}}] 
Since the forbidden move is an unknotting operation 
for virtual knots~\cite{Kan,Nel}, 
we have the conclusion by Lemma~\ref{lem-forbidden}. 
\end{proof}

\begin{remark} 
It is proved in \cite{YI} that 
the forbidden detour move is also an unknotting operation for 
virtual knots. 
By using this result, 
Theorem~\ref{thm13} is  
a direct consequence of Lemma~\ref{lem-fd}. 
\end{remark}

By Theorem~\ref{thm13}, 
any two virtual knots $K$ and $K'$ are $v\Delta$-equivalent to each other. 
We denote by ${\rm d}_{v\Delta}(K,K')$ 
the minimal number of virtualized $\Delta$-moves 
needed to deform a diagram of $K$ into that of $K'$. 
It is called the \textit{$v\Delta$-distance} between $K$ and $K'$. 
In particular, 
we denote ${\rm d}_{v\Delta}(K,O)$ by ${\rm u}_{v\Delta}(K)$, 
and call it the {\it $v\Delta$-unknotting number} of $K$, 
where $O$ is the trivial knot.

We briefly recall the definitions of the $n$-writhe $J_n(K)$ 
and the odd writhe $J(K)$ 
of a virtual knot $K$ from~\cite{ST}. 
Let $G$ be a Gauss diagram of $K$ and $\gamma$ a chord of $G$.  
If $\gamma$ admits a sign $\varepsilon$, 
we assign $\varepsilon$ and $-\varepsilon$ 
to the terminal and initial endpoints of $\gamma$, respectively. 
The endpoints of $\gamma$ divide the underlying oriented circle of $G$ 
into two arcs. 
Let $\alpha$ be the one of the two oriented arcs 
which runs from the initial endpoint of $\gamma$ to the terminal. 
The {\it index} of $\gamma$ 
is the sum of the signs of all the endpoints of chords on $\alpha$, 
and denoted by ${\rm Ind}(\gamma)$. 
For a nonzero integer $n$, the {\it $n$-writhe} $J_{n}(K)$ of $K$ is the sum of the signs 
of all the chords $\gamma$ with ${\rm Ind}(\gamma)=n$, 
and the {\it odd writhe} $J(K)$ of $K$ is defined to be $\sum_{n{\rm :odd}}J_n(K)$.
We remark that $J(K)$ is always even~\cite{Che}.

We give a lower bound for ${\rm d}_{v\Delta}(K,K')$ 
by using the odd writhes of $K$ and $K'$ 
as follows.

\begin{proposition}\label{prop-distance}
Let $K$ and $K'$ be virtual knots. 
Then we have the following. 
\begin{itemize}
\item[{\rm (i)}] 
${\rm d}_{v\Delta}(K,K')\geq \frac{1}{2}|J(K)-J(K')|$. 
\item[{\rm (ii)}] 
${\rm u}_{v\Delta}(K)\geq \frac{1}{2}|J(K)|$. 
\end{itemize}
\end{proposition}

\begin{proof}
Let $G$ and $G'$ be 
Gauss diagrams of $K$ and $K'$, respectively. 
It is sufficient to prove that 
if $G$ and $G'$ are related 
by a virtualized $\Delta$-move, 
then the odd writhes satisfy 
$J(K)-J(K')\in\{0,\pm 2\}$. 

We may assume that $G'$ is obtained from $G$ 
by removing three chords corresponding to 
three real crossings 
involved in a virtualized $\Delta$-move. 
See Figure~\ref{pf-prop-distance}. 
It is known that among these three chords, 
the number of chords with index odd 
is equal to zero or two (cf.~\cite{Man, Tur}). 
Furthermore, the parity of the index of any other chord is preserved  
by the virtualized $\Delta$-move; 
in fact, each pair of the three chords 
has two adjacent endpoints on the underlying circle. 
See the figure again. 
Therefore we have the conclusion. 
\end{proof}

\begin{figure}[htbp]
  \centering
    \begin{overpic}[width=6cm]{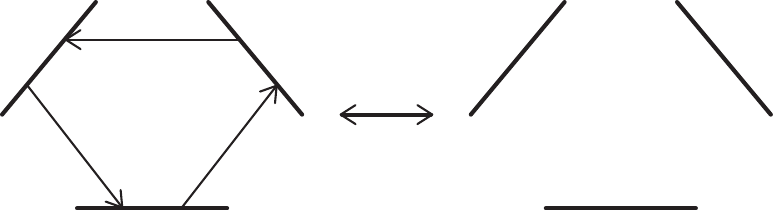}
      \put(79,26){$v\Delta$}
      \put(30,-15){$G$}
      \put(133,-15){$G'$}
    \end{overpic}
   \vspace{1em}
  \caption{A virtualized $\Delta$-move on a Gauss diagram}
  \label{pf-prop-distance}
\end{figure}

It is known that 
the crossing change at a real crossing 
is not an unknotting operation for virtual knots (cf.~\cite{CKS,HK,Tur}).  
If a virtual knot $K$ can be deformed into the trivial knot 
by a finite number of crossing changes, 
then we denote by ${\rm u}(K)$ the minimal number of such 
crossing changes. 
If $K$ cannot be unknotted by crossing changes, 
then we set ${\rm u}(K)=\infty$. 
Then we have the following by Lemma~\ref{lem-cc} immediately.

\begin{lemma}\label{lem-ucc}
Any virtual knot $K$ satisfies 
${\rm u}_{v\Delta}(K)\leq {\rm u}(K)$. 
\hfill$\Box$ 
\end{lemma}

In the following, 
we will construct two families of infinitely many virtual knots $K$ 
with ${\rm u}_{v\Delta}(K)=m$ for any given integer $m\geq 1$, 
which have different properties for ${\rm u}(K)$. 
The following theorems induce Theorem~\ref{thm14} 
immediately. 

\begin{theorem}\label{thm-infinite}
For any integer $m\geq 1$, 
there are infinitely many virtual knots $K$ 
with ${\rm u}_{v\Delta}(K)={\rm u}(K)=m$. 
\end{theorem}

\begin{proof}
For an integer $s\geq1$, 
we consider a knot diagram and its Gauss diagram 
with $2m+2s-1$ real crossings (or chords) 
as shown in Figure~\ref{pf-thm-infinite}. 
Let $K_{s}(m)$ be the virtual knot presented by this diagram. 

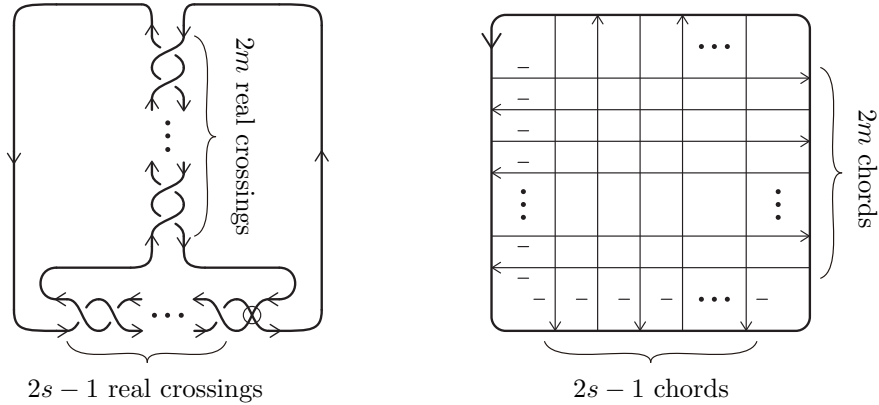
\begin{figure}[htbp]
  \centering
    \begin{overpic}[width=11cm]{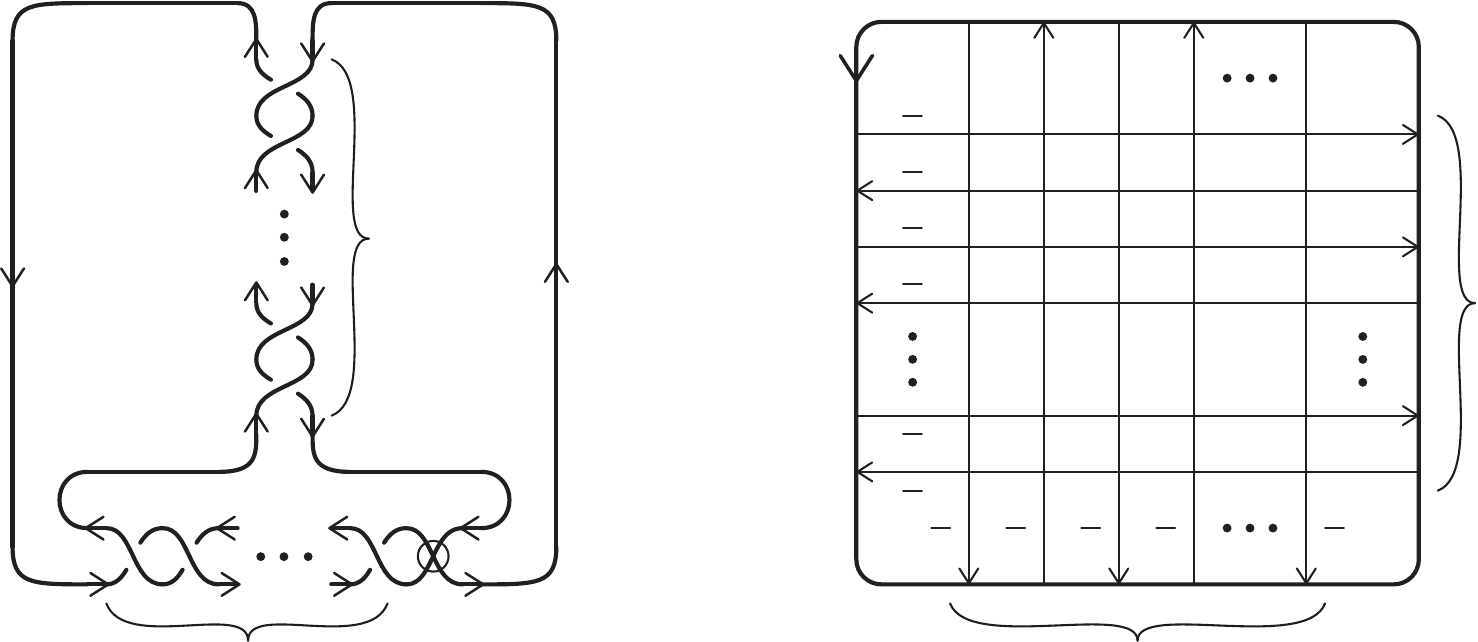}
      \put(84,123){\rotatebox{-90}{$2m$ real crossings}}
      \put(7.7,-13){$2s-1$ real crossings}
      \put(319,96){\rotatebox{-90}{$2m$ chords}}
      \put(212,-13){$2s-1$ chords}
    \end{overpic}
  \vspace{1em}
  \caption{A diagram of $K_{s}(m)$ and its Gauss diagram}
  \label{pf-thm-infinite}
\end{figure}

Since $K_s(m)$ can be unknotted by $m$ crossing changes 
at $m$ real crossings among 
$2m$ half twists in the knot diagram, 
we have ${\rm u}(K_s(m))\leq m$. 
By Lemma~\ref{lem-ucc}, we have 
\[
{\rm u}_{v\Delta}(K_s(m))
\leq {\rm u}(K_s(m))\leq m. 
\]

On the other hand, 
the $2m$ horizontal chords of the Gauss diagram 
have the sign $-1$ and indices $\pm 1$, and 
the $2s-1$ vertical chords have index $0$. 
Since $J(K_s(m))=-2m$ holds, 
we have 
${\rm u}_{v\Delta}(K_s(m))\geq m$ by Proposition~\ref{prop-distance}(ii), 
and hence ${\rm u}_{v\Delta}(K_s(m))={\rm u}(K_s(m))=m$.

Now it is enough to prove that the virtual knots $K_s(m)$'s are mutually distinct. 
By a straightforward calculation, 
the Jones polynomial $f_{K_s(m)}(A)\in{\Z}[A,A^{-1}]$ of $K_s(m)$ 
is given by 
$$f_{K_s(m)}(A)=A^{8m}+(A^{-4}-A^{-8})
\left(\sum_{i=1}^m A^{8i}\right)
\left(-A^{8s-2}+\sum_{j=1}^{2s-1}(-1)^jA^{4j}\right).$$ 
Since the maximal degree of $f_{K_s(m)}(A)$ is equal to 
$8m+8s-6$, 
we have $K_s(m)\ne K_{s'}(m)$ for any $s\ne s'$. 
\end{proof}

\begin{theorem}\label{thm-infinite2}
For any integer $m\geq 1$, 
there are infinitely many virtual knots $K$ 
with ${\rm u}_{v\Delta}(K)=m$ and ${\rm u}(K)=\infty$. 
\end{theorem}

\begin{proof}
For an integer $s\geq 2$, 
we consider a long virtual knot diagram 
and its Gauss diagram with 
$2s+3$ real crossings (or chords) 
$a_1,a_2,a_3$ and $b_1,\dots,b_{2s}$ 
as shown in 
Figure~\ref{fig-infinite-example}. 
Let $T_s$ be the long virtual knot 
presented by this diagram, 
and $K_s(m)$ the virtual knot obtained from 
the closure of the product of $m$ copies of $T_s$. 

Since $T_s$ can be unknotted by 
a single virtualized $\Delta$-move involving the three crossings 
$a_{1}$, $a_{2}$, and $a_{3}$ in the long knot diagram, 
we have ${\rm u}_{v\Delta}(K_s(m))\leq m$.

On the other hand, 
since we have 
$${\rm Ind}(a_1)=1, \ {\rm Ind}(a_2)=2s, \ 
{\rm Ind}(a_3)=-2s-1, \mbox{ and }
{\rm Ind}(b_i)=2 \ (1\leq i\leq 2s),$$
it follows from \cite[Lemma 4.3]{ST} that 
\[
J_{n}(K_{s}(m))=
\begin{cases}
-m & \text{if } n=2s, \\
2ms & \text{if } n=2, \\
m & \text{if } n=1,-2s-1, \\ 
0 & \text{otherwise}.  
\end{cases}
\]
This induces 
$J(K_s(m))=J_1(K_s(m))+J_{-2s-1}(K_s(m))=2m$. 
Therefore we have ${\rm u}_{v\Delta}(K_s(m))\geq m$ 
by Proposition~\ref{prop-distance}(ii), 
and hence ${\rm u}_{v\Delta}(K_s(m))=m$. 
Furthermore, 
since $J_1(K_s(m))=m\ne 0=J_{-1}(K_s(m))$ holds, 
we have ${\rm u}(K_s(m))=\infty$ by \cite[Theorem~1.5]{ST}. 

Since $J_2(K_s(m))=2ms$ holds, 
we have $K_s(m)\ne K_{s'}(m)$ for any $s\ne s'$. 
\end{proof}

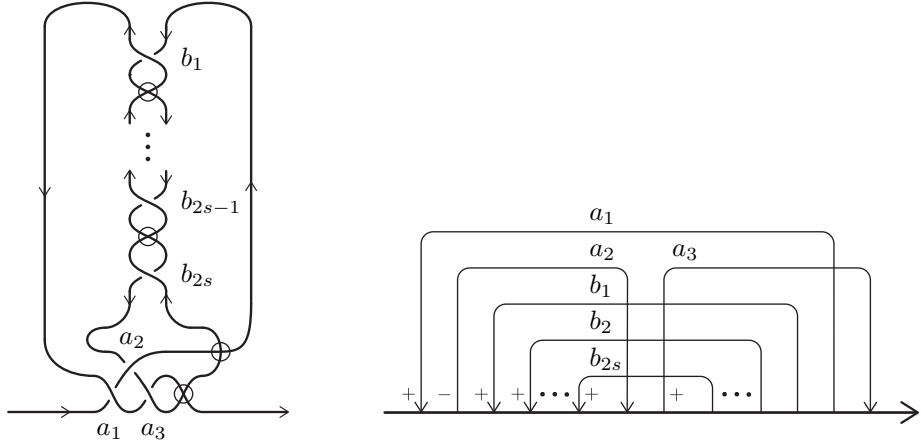
\begin{figure}[htbp]
  \centering
    \begin{overpic}[width=12cm]{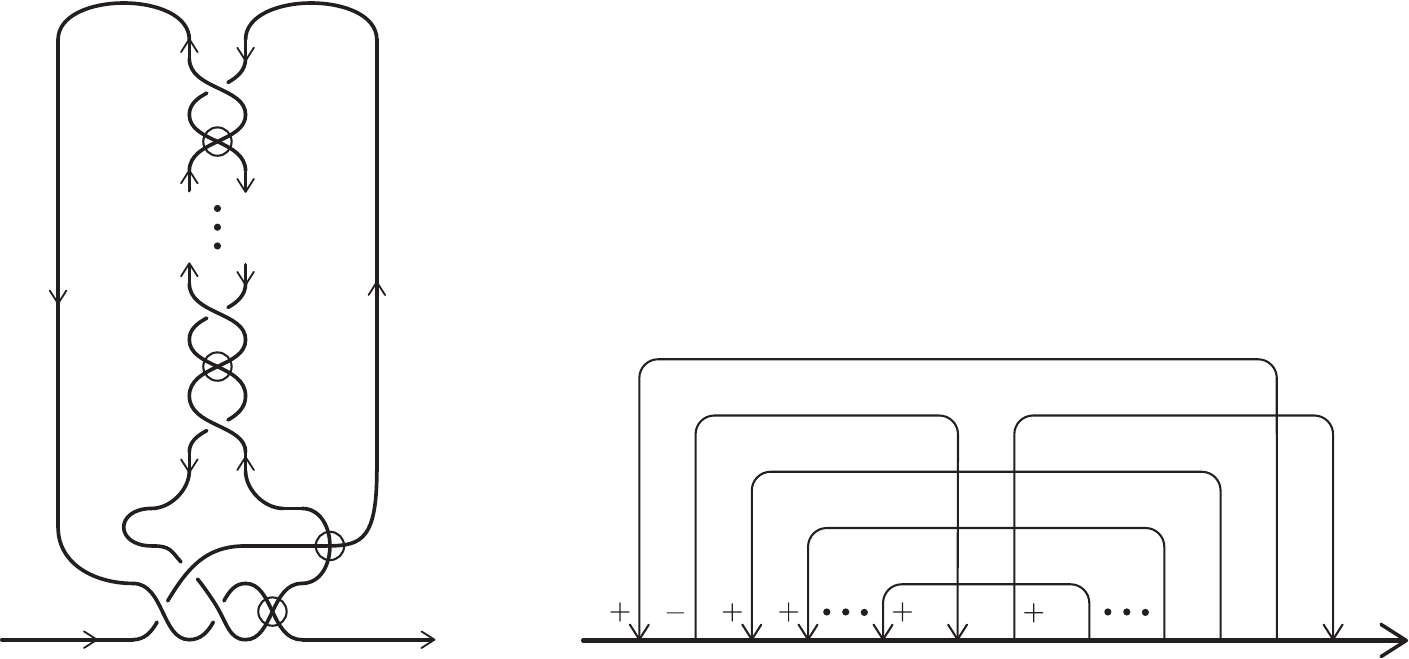}
      \put(33.5,-4){$a_{1}$}
      \put(50,-4){$a_{3}$}
      \put(42,30){$a_{2}$}
      \put(65,135){$b_{1}$}
      \put(65,81){$b_{2s-1}$}
      \put(65,53.5){$b_{2s}$}
      \put(218,77){$a_{1}$}
      \put(218,63.5){$a_{2}$}
      \put(249,63.5){$a_{3}$}
      \put(218,49.5){$b_{1}$}
      \put(218,36){$b_{2}$}
      \put(218,22){$b_{2s}$}
    \end{overpic}
  \vspace{1em}
  \caption{A diagram of $T_s$ and its Gauss diagram}
  \label{fig-infinite-example}
\end{figure}

\section{Virtualized $\Delta$-moves for virtual links}\label{sec3}

\begin{lemma}\label{lem-or}
A local deformation as shown in {\rm Figure~\ref{orientation-reversal}} 
is realized by a combination of 
a virtualized $\Delta$-move and 
generalized Reidemeister moves. 
\end{lemma}

\begin{figure}[htbp]
  \centering
    \begin{overpic}[width=6cm]{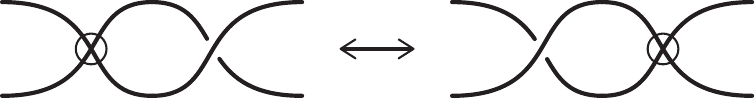}
    \end{overpic}
  \caption{A local deformation considered in Lemma~\ref{lem-or}}
  \label{orientation-reversal}
\end{figure}

\begin{proof}
This follows from Figure~\ref{pf-lem-or}.
\end{proof}

\begin{figure}[htbp]
  \centering
    \begin{overpic}[width=12cm]{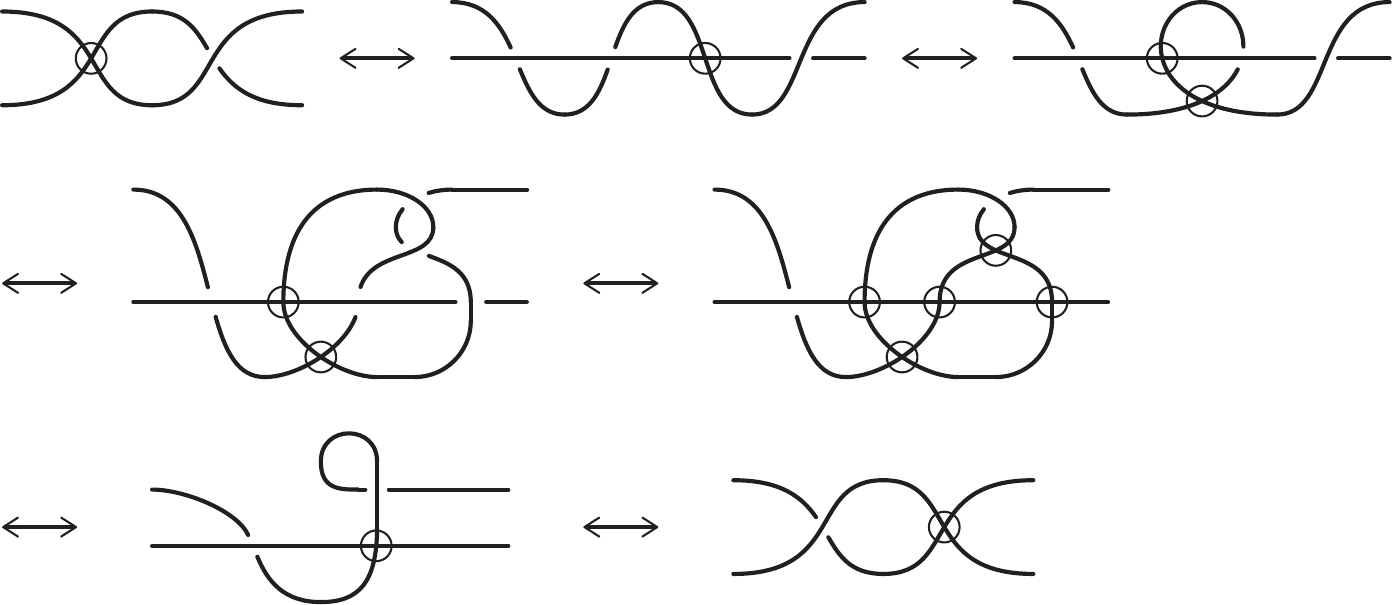}
      \put(89,138){R}
      \put(227,138){R}
      \put(6.3,83){R}
      \put(145.5,84){$v\Delta$}
      \put(6.3,23){R}
      \put(148.8,23){R}
    \end{overpic}
  \caption{Proof of Lemma~\ref{lem-or}}
  \label{pf-lem-or}
\end{figure}

\begin{lemma}\label{lem-reverse}
Let $\gamma$ be a chord of a Gauss diagram $G$. 
\begin{itemize}
\item[{\rm (i)}] 
If a Gauss diagram $G'$ is obtained from $G$ 
by reversing the orientation of $\gamma$, 
then $G$ and $G'$ are related by a finite sequence of 
virtualized $\Delta$-moves and Reidemeister moves. 
\item[{\rm (ii)}] 
If a Gauss diagram $G''$ is obtained from $G$ 
by changing the sign of $\gamma$, 
then $G$ and $G''$ are related by a finite sequence of 
virtualized $\Delta$-moves and Reidemeister moves. 
\end{itemize}
\end{lemma}

\begin{proof}
(i) This follows from Lemma~\ref{lem-or}. 

(ii) Let $G'''$ be a Gauss diagram obtained from $G$ 
by reversing the orientation and changing the sign of $\gamma$. 
By Lemma~\ref{lem-cc}, $G$ and $G'''$ are related 
by a finite sequence of a virtualized $\Delta$-move 
and Reidemeister moves. 
Furthermore $G''$ is obtained from $G'''$ 
by reversing the orientation of $\gamma$. 
Therefore we have the conclusion by (i). 
\end{proof}

\begin{lemma}\label{lem-switch}
Let $\gamma$ and $\gamma'$ be chords of a Gauss diagram $G$ 
such that an endpoint of $\gamma$ is adjacent to 
that of $\gamma'$. 
If a Gauss diagram $G'$ is obtained from $G$ 
by switching the positions of these consecutive endpoints, 
then $G$ and $G'$ are related by a finite sequence of 
virtualized $\Delta$-moves and Reidemeister moves. 
\end{lemma}

\begin{proof}
This follows from Lemmas~\ref{lem-fd} and \ref{lem-forbidden}. 
\end{proof}

For $n-1$ integers $a_2,\dots,a_n\in\{0,1\}$ with $n\geq 2$, 
let $H(a_2,\dots,a_n)=\bigcup_{i=1}^n H_i$ be the Gauss diagram 
of an $n$-component virtual link 
such that 
\begin{itemize}
\item[(i)] 
$H(a_2,\dots,a_n)$ has no self-chords, 
\item[(ii)] 
there are no nonself-chords between $H_i$ and $H_j$ 
$(2\leq i< j\leq n)$, 
\item[(iii)] 
if $a_i=0$, then there are no nonself-chords 
between $H_1$ and $H_i$, 
\item[(iv)] 
if $a_i=1$, then there is a single nonself-chord 
between $H_1$ and $H_i$ 
which is oriented from $H_1$ to $H_i$ 
with a positive sign, and 
\item[(v)] 
if $a_i=a_j=1$ $(2\leq i<j\leq n)$, 
then we meet the endpoint of the chord between $H_1$ and $H_i$ 
before that between $H_1$ and $H_j$ along $H_1$. 
\end{itemize}
Figure~\ref{ex-H} shows the Gauss diagram $H(1,0,1,1,0)$ with $n=6$. 
Let $M(a_2,\dots,a_n)$ be the $n$-component 
virtual link presented by $H(a_2,\dots,a_n)$. 

\begin{figure}[htbp]
  \centering
    \begin{overpic}[width=7cm]{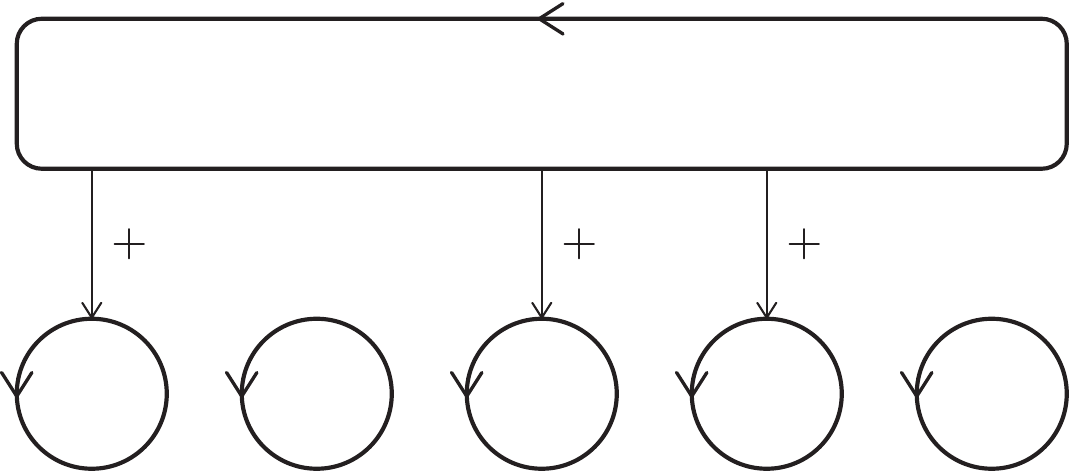}
      \put(-15,67){$H_{1}$}
      \put(10,-13){$H_{2}$}
      \put(52,-13){$H_{3}$}
      \put(94,-13){$H_{4}$}
      \put(136,-13){$H_{5}$}
      \put(178,-13){$H_{6}$}
    \end{overpic}
  \vspace{1em}
  \caption{The Gauss diagram $H(1,0,1,1,0)$}
  \label{ex-H}
\end{figure}
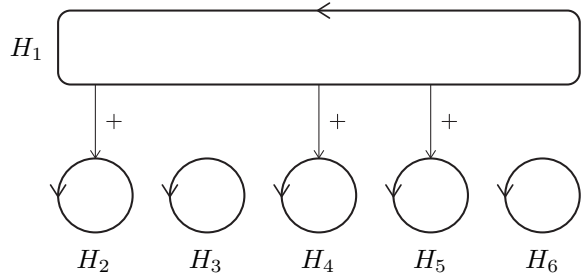

\begin{proposition}\label{prop-form}
Let $n\geq2$ be an integer. 
Any $n$-component virtual link $L=\bigcup_{i=1}^n K_i$ 
is $v\Delta$-equivalent to 
$M(a_2,\dots,a_n)$ for some $a_i$'s. 
\end{proposition}

\begin{proof}
Let $G=\bigcup_{i=1}^n G_i$ be a Gauss diagram of $L$. 
By using Lemma~\ref{lem-switch} and 
Reidemeister moves I, 
we can remove all the self-chords from $G$. 
Therefore we may assume that $G$ has 
no self-chords 
(up to virtualized $\Delta$-moves and Reidemeister moves). 

If there is a nonself-chord 
between $G_i$ and $G_j$ $(2\leq i< j \leq n)$, 
then we can replace it with 
a pair of nonself-chords 
one of which connects between $G_1$ and $G_i$, 
and the other between $G_1$ and $G_j$. 
In fact, this is achieved by observing a sequence of link diagrams 
as shown in Figure~\ref{pf-prop-form}. 
Therefore we may also assume that 
there are no nonself-chords 
between $G_i$ and $G_j$ $(2\leq i< j\leq n)$. 

\begin{figure}[htbp]
  \centering
  \vspace{1em}
    \begin{overpic}[width=8cm]{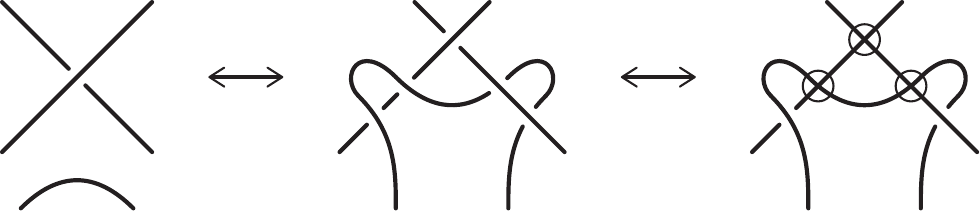}
      \put(-12,53){$K_{j}$}
      \put(33,53){$K_{i}$}
      \put(13,-13){$K_{1}$}
      \put(54,35.5){R}
      \put(84,53){$K_{j}$}
      \put(112,53){$K_{i}$}
      \put(100,-13){$K_{1}$}
      \put(147,35.5){$v\Delta$}
      \put(179,53){$K_{j}$}
      \put(208,53){$K_{i}$}
      \put(196,-13){$K_{1}$}
    \end{overpic}
  \vspace{1em}
  \caption{A sequence of link diagrams}
  \label{pf-prop-form}
\end{figure}
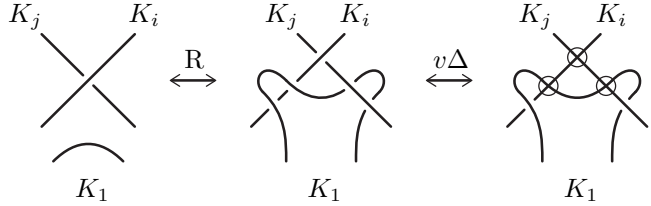

By using Lemmas~\ref{lem-reverse}, \ref{lem-switch}, 
and Reidemeister moves II, 
we can reduce the number of nonself-chords 
between $G_1$ and $G_i$ $(2\leq i\leq n)$ 
to zero or one. 

Finally, any Gauss diagram $G$ of $L$ 
can be deformed into $H(a_2,\dots,a_n)$ for some $a_i$'s 
by Lemmas~\ref{lem-reverse} and \ref{lem-switch}. 
\end{proof}

For $n\geq 2$, 
let $L=K_1\cup\dots\cup K_n$ be an $n$-component virtual link, 
and $G=G_1\cup\dots\cup G_n$ a Gauss diagram of $L$. 
For each $i=1,\dots,n$, 
the {\it $i$th parity} of $L$ is 
defined to be the parity of the number of 
endpoints of nonself-chords on $G_i$, 
and denoted by $p_i(L)\in\{0,1\}$. 
Since a self-chord has two endpoints on the same component, 
$p_i(L)$ is coincident with the parity of 
the number of endpoints of 
self-/nonself-chords on $G_i$.

\begin{lemma}\label{lem-parity}
For any $i=1,\dots,n$, 
the $i$th parity $p_i(L)$ is an invariant of 
the $v\Delta$-equivalence class of $L$. 
\end{lemma}

\begin{proof}
Since the chord(s) involved in a Reidemeister move 
or a virtualized $\Delta$-move 
have an even number of endpoints on each $G_i$, 
this move preserves $p_i(L)$. 
\end{proof}

\begin{remark}
For $n\geq 2$, any $n$-component virtual link $L$ satisfies 
$$p_1(L)+p_2(L)+\dots+p_n(L)\equiv 0 \ ({\rm mod}~2).$$
In fact, the number of endpoints of 
all self-/nonself-chords is even. 
\end{remark}

\begin{proof}[Proof of {\rm Theorem~\ref{thm15}}] 
The only if part follows from Lemma~\ref{lem-parity}. 

We will prove the if part. 
By Proposition~\ref{prop-form}, 
$L$ and $L'$ are $v\Delta$-equivalent to 
$M(a_2,\dots,a_n)$ and $M(b_2,\dots,b_n)$ for some 
$a_i$'s and $b_i$'s, respectively. 
By assumption and Lemma~\ref{lem-parity}, 
it holds that 
$$a_i=p_i\bigl(M(a_2,\dots,a_n)\bigr)=p_i(L)=p_i(L')
=p_i\bigl(M(b_2,\dots,b_n)\bigr)=b_i$$ 
for any $i=1,\dots,n$. 
Therefore $L$ and $L'$ are $v\Delta$-equivalent to 
$M(a_2,\dots,a_n)=M(b_2,\dots,b_n)$. 
\end{proof}

The following is an immediate consequence of 
the proof of Theorem~\ref{thm15}. 

\begin{corollary}
For $n\geq 2$, 
a complete representative system of 
the $v\Delta$-equivalence classes 
of $n$-component virtual links is given by 
$$\left\{M(a_2,\dots,a_n)\mid a_2,\dots,a_n\in\{0,1\}\right\}.$$
In particular, 
the number of $v\Delta$-equivalence classes 
is equal to $2^{n-1}$. 
\hfill$\Box$
\end{corollary}

\section{Relations among four types of virtualized $\Delta$-moves}\label{sec4}

We can divide virtualized $\Delta$-moves into 
four types $rv\Delta$, $vr\Delta$, $rv\Delta'$, and $vr\Delta'$ 
as shown in Figure~\ref{4types}. 
Here, 
the $rv\Delta$- and $rv\Delta'$-moves 
replace three real crossings with virtual ones, 
and the $vr\Delta$- and $vr\Delta'$-moves 
replace three virtual crossings with real ones. 

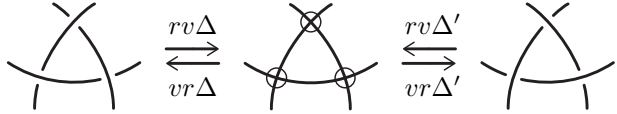
\begin{figure}[htbp]
  \centering
    \begin{overpic}[width=8cm]{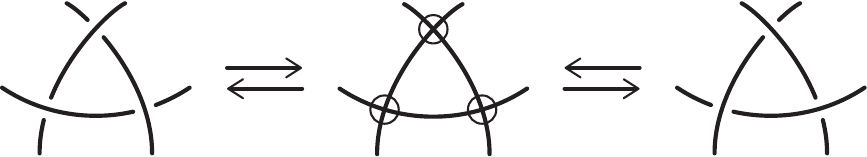}
      \put(60.5,29){$rv\Delta$}
      \put(60.5,5.5){$vr\Delta$}
      \put(149,29){$rv\Delta'$}
      \put(149,5.5){$vr\Delta'$}
    \end{overpic}
  \caption{Four types $rv\Delta$, $vr\Delta$, $rv\Delta'$, and $vr\Delta'$}
  \label{4types}
\end{figure}

\begin{lemma}\label{lem-ccc}
For any $X\in\{rv\Delta, vr\Delta, rv\Delta', vr\Delta'\}$, 
a crossing change at a real crossing is realized by a combination of 
an $X$-move and generalized Reidemeister moves. 
\end{lemma}

\begin{proof}
For $X=rv\Delta$, 
a crossing change at a real crossing is obtained from the deformation sequence in Figure~\ref{pf-lem-cc} 
by replacing $\stackrel{v\Delta}{\longleftrightarrow}$ 
with $\stackrel{rv\Delta}{\longrightarrow}$. 
For $X=vr\Delta$, 
we may follow this sequence in reverse. 
For $X=rv\Delta'$ and $vr\Delta'$, 
we may perform the crossing change at every real crossing 
in the above sequences for $rv\Delta$ and $vr\Delta$, respectively. 
\end{proof}

\begin{lemma}\label{lem-sr}
A local deformation SR 
as shown in {\rm Figure~\ref{sign-reversal}}
is realized by a combination of $rv\Delta$-moves and 
generalized Reidemeister moves. 
\end{lemma}

\begin{figure}[htbp]
  \centering
    \begin{overpic}[width=6cm]{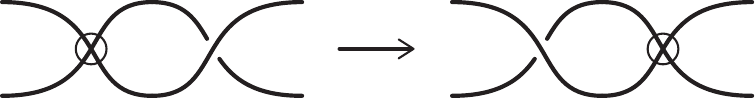}
      \put(78.5,16){SR}
    \end{overpic}
  \caption{A local deformation SR}
  \label{sign-reversal}
\end{figure}

\begin{proof}
Figure~\ref{pf-lem-sr} indicates the proof. 
More precisely, 
the first deformation in this figure 
is obtained from the deformation sequence in Figures~\ref{pf-lem-or} 
by replacing $\stackrel{v\Delta}{\longleftrightarrow}$ 
with $\stackrel{rv\Delta}{\longrightarrow}$. 
The second deformation 
is a crossing change which is realized by an $rv\Delta$-move 
and generalized Reidemeister moves 
by Lemma~\ref{lem-ccc}. 
\end{proof}

\begin{figure}[htbp]
  \centering
    \begin{overpic}[width=10cm]{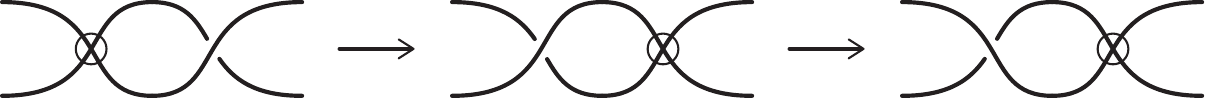}
      \put(190,16){cc}
    \end{overpic}
  \caption{Proof of Lemma~\ref{lem-sr}}
  \label{pf-lem-sr}
\end{figure}

\begin{remark} 
The local deformation SR in Figure~\ref{sign-reversal} 
is called a {\it sign reversal move}~\cite{ABMW}. 
\end{remark}

\begin{theorem}\label{thm-relation}
For any $X\ne Y\in\{rv\Delta, vr\Delta, rv\Delta', vr\Delta'\}$, 
a $Y$-move is realized by a combination of 
$X$-moves and generalized Reidemeister moves. 
\end{theorem}

\begin{proof} 
We use the notation ``$X\Rightarrow Y$'' if a $Y$-move is 
realized by a combination of $X$-moves 
and generalized Reidemeister moves. 
Then we have the following. 

\begin{itemize}
\item
$rv\Delta\Rightarrow vr\Delta$: 
The sequence in Figure~\ref{pf-thm-relation} shows that 
a $vr\Delta$-move is realized by a combination of 
an $rv\Delta$-move, three sign reversal moves, 
and several generalized Reidemeister moves. 
Therefore we have $rv\Delta\Rightarrow vr\Delta$ by Lemma~\ref{lem-sr}. 

\item
$vr\Delta\Rightarrow vr\Delta'$: 
The sequence in Figure~\ref{pf-thm-relation2} shows that 
a $vr\Delta'$-move is realized by a $vr\Delta$-move and three crossing changes. 
Therefore we have $vr\Delta\Rightarrow vr\Delta'$ by Lemma~\ref{lem-ccc}.

\item
$vr\Delta'\Rightarrow rv\Delta'$:
We may follow any sequence for $rv\Delta\Rightarrow vr\Delta$ 
with opposite crossing information in reverse. 
\item
$rv\Delta'\Rightarrow rv\Delta$:
We may follow any sequence for $vr\Delta\Rightarrow vr\Delta'$ 
with opposite crossing information in reverse. 
\end{itemize}
Therefore we have the cycle 
$$rv\Delta\Rightarrow vr\Delta\Rightarrow 
vr\Delta'\Rightarrow rv\Delta'\Rightarrow rv\Delta,$$
which induces the other eight cases immediately. 
\end{proof}

\begin{figure}[htbp]
  \centering
    \begin{overpic}[width=12cm]{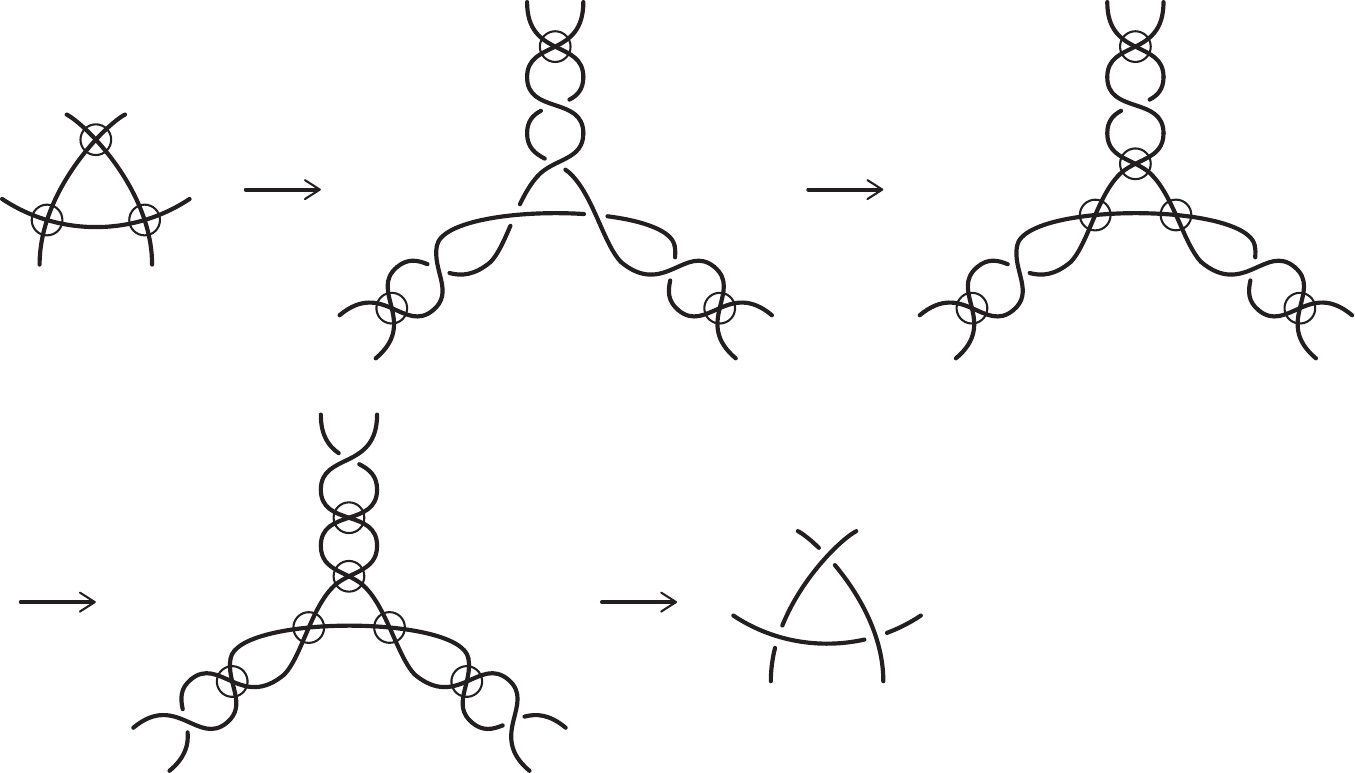}
      \put(67,152){R}
      \put(204,152){$rv\Delta$}
      \put(7,48){SR}
      \put(157,48){R}
    \end{overpic}
  \caption{A $vr\Delta$-move is realized by $rv\Delta$-moves}
  \label{pf-thm-relation}
\end{figure}

\begin{figure}[htbp]
  \centering
    \begin{overpic}[width=9cm]{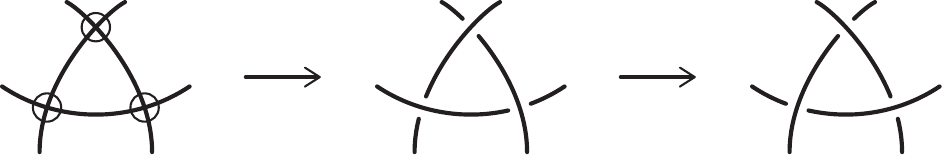}
      \put(67.5,26){$vr\Delta$}
      \put(173.5,26){cc}
    \end{overpic}
  \caption{A $vr\Delta'$-move is realized by $vr\Delta$-moves}
  \label{pf-thm-relation2}
\end{figure}


\end{document}